\documentclass[letterpaper, 10 pt, conference]{ieeeconf}  

\IEEEoverridecommandlockouts                              
\usepackage{cite,xspace}
\usepackage{graphicx} 
\graphicspath{ {./images/} }
\usepackage{algorithm}
\usepackage{algpseudocode}
\usepackage{tikz}
\usepackage{amsmath} 
\usepackage{amssymb}  
\DeclareMathOperator*{\argmin}{arg\,min}

\usepackage{xcolor}
\makeatletter
\let\NAT@parse\undefined
\makeatother
\usepackage[pdftex,pdfpagelabels,bookmarks,hyperindex,hyperfigures]{hyperref}
\hypersetup{colorlinks,linkcolor={red!50!black},citecolor={blue!50!black},urlcolor={blue!80!black}}
\usepackage[capitalise]{cleveref}
\usepackage[mathscr]{euscript}

\usepackage{amsthm}
\theoremstyle{definition}
\newtheorem{theorem}{Theorem}
\newtheorem*{problem}{Problem}
\newtheorem{lemma}{Lemma}
\newtheorem*{defn}{Definition}

\DeclareMathOperator{\seq}{seq}
\DeclareMathOperator{\rank}{rank}
\newcommand{\mT}{\mathcal{T}}
\newcommand{\mD}{\mathcal{D}}
\newcommand{\mU}{\mathcal{U}}
\newcommand{\mC}{\mathcal{C}}
\newcommand{\mJ}{\mathcal{J}}
\newcommand{\Tm}{T_{\min}}
\newcommand{\TM}{T_{\max}}
\newcommand{\um}{u_{\min}}
\newcommand{\uM}{u_{\max}}
\newcommand{\com}{Center of Mass\xspace}
\newcommand{\cop}{Center of Pressure\xspace}
\newcommand{\XY}{(x_1,x_2)}
\newcommand{\bounded}{\mathcal{B}{d}}
\newcommand{\region}[1]{\mathcal{R}_{#1}}

\newcommand{\fT}{\mathscr T}

\newcommand{\alignedintertext}[1]{%
  \noalign{%
    \vskip\belowdisplayshortskip
    \vtop{\hsize=\linewidth#1\par
    \expandafter}%
    \expandafter\prevdepth\the\prevdepth
  }
}

\title{\LARGE \bf
Fast replanning of a lower-limb exoskeleton trajectories for rehabilitation }
\author{Maxime Brunet, Marine Pétriaux, Florent Di Meglio and Nicolas Petit \thanks{M. Brunet and M. Pétriaux are with Wandercraft, 88 Rue de Rivoli, 75004 Paris, France {\tt\small maxime.brunet@wandercraft.eu}}
 \thanks{F. Di Meglio and N. Petit are with MINES Paris, Centre Automatique et Systèmes, PSL University, 60 bd. St Michel, 75272 Paris Cedex, France}
 }

\begin{document}
\maketitle
\thispagestyle{empty}
\pagestyle{empty}

\begin{abstract}
 The paper addresses the rehabilitation of disabled patients using a lower-limb fully-actuated exoskeleton. We propose a novel numerical method to replan the current step without jeopardizing stability. Stability is evaluated in the light of a simple linear time-invariant surrogate model. The method's core is the analysis of an input-constrained optimal control problem with state specified at an unspecified terminal time. A detailed study of the extremals given by Pontryagin Maximum Principle is sufficient to characterize its feasibility. This allows a fast replanning strategy. The efficiency of the numerical algorithm (resolution time below 1\,ms) yields responsiveness to the patient's request. Realistic simulations on a full-body model of the patient-exoskeleton system stress that cases of practical interest for physiotherapists are well-addressed. 
 \end{abstract}
\section{INTRODUCTION}
Exoskeletons have been proposed for various tasks  since the 1970s~\cite{b01,b02}.
In particular, lower-limb exoskeletons are now being developed for gait rehabilitation~\cite{b03}.
Various exoskeleton technologies are being considered such as ground-tethered~\cite{b04,b06,b07}, crutches-aided~\cite{b08, b09, b010, b011}, or self-balancing exoskeletons as in the fully-actuated exoskeleton Atalante by Wandercraft~\cite{b014}. In the latter case, stability during walking is achieved by closed-loop controllers.

Rehabilitation is a task of high medical interest during which the patient cooperates with the actuators of the exoskeleton (also referred to as robot) and provides a substantial part of the mechanical effort. This has a strong therapeutic effect as it allows to train lost body functions.
The degree of effort sharing can be tuned by a physiotherapist, according to the patient capabilities and desired level of training. In a self-balancing exoskeleton, the patient is guided by low and high-level controllers, in charge of stabilization tasks, maintaining the system in the vicinity of pre-defined geometric paths corresponding to nominal walking gaits. Instead of being traveled at their nominal velocity, the duration of each step of the gait can be freely adjusted to reward the patient's efforts. However, this adjustment is not an easy task. Using a simple time scaling to change the velocity at which the step is executed may result in an unstable walking gait. Experimentally, it is observed that the robot often falls if the velocity is kept below some threshold for a sufficient time. 
\begin{figure}[ht]
   \centering
   \includegraphics[scale=0.15]{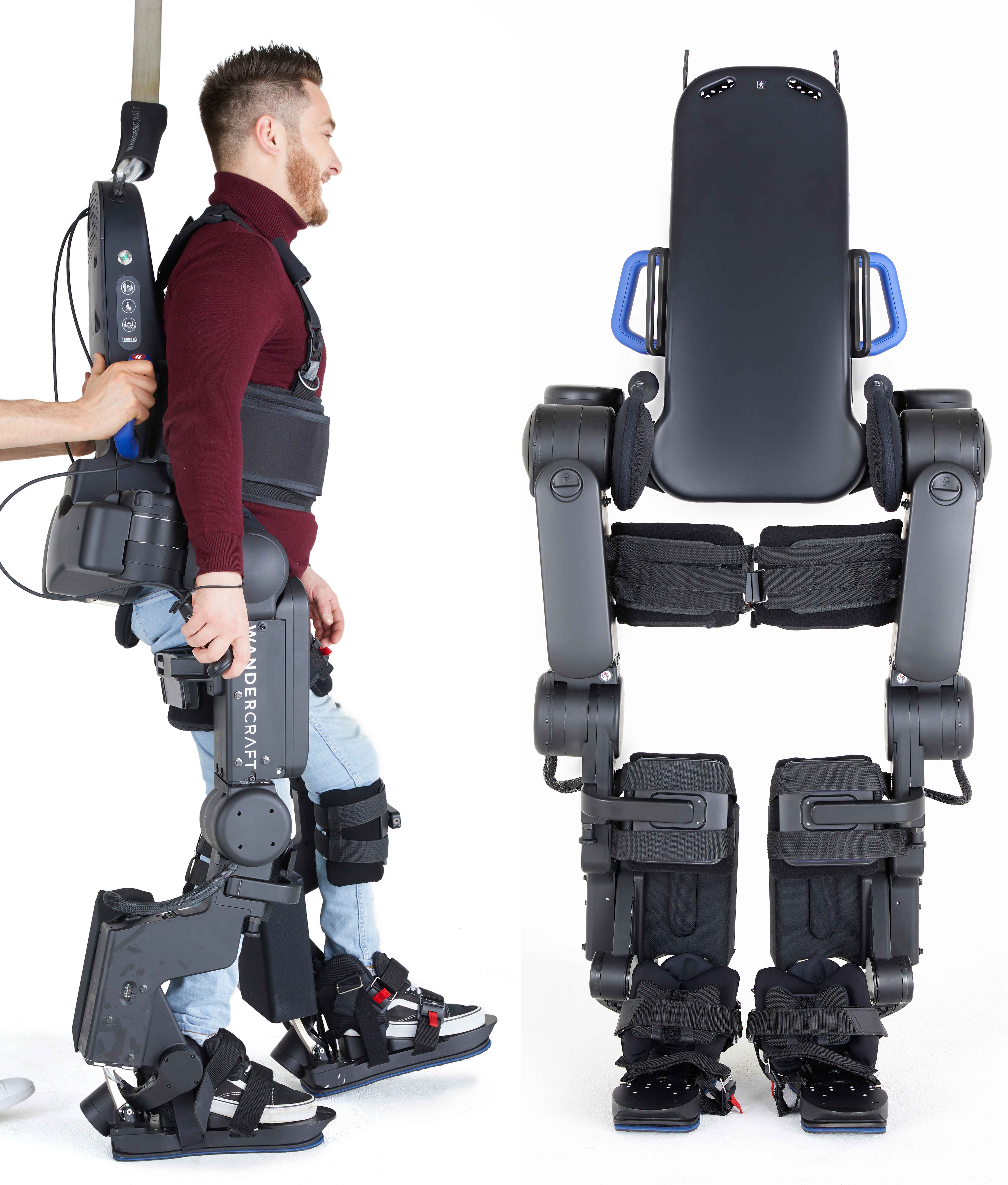}
    \caption{A patient walks using an Atalante exoskeleton under the guidance of a physiotherapist during a rehabilitation exercice.}
   \label{fig:exo}
\end{figure}

At nominal velocity, stability of trajectories is ensured by the carefully tuned low and high-level controllers mentioned above. A point worth noticing is that, for the actuation to have the desired effect on the system, the trajectory's \cop \emph{must} lie inside the \emph{support polygon}. By design, this condition holds at all times when nominal velocity is used. However, when the system trajectory is accelerated or slowed down by too large factors, an inverse dynamics calculus reveals that the \cop leaves the support polygon by large amounts. This is the root cause of the observed instability.
To address this issue, two natural ideas come to mind. The first one is to generate a large  library of trajectories, offline, corresponding to a wide variety of step durations. While doable in principle, it would require particularly extensive efforts to cover all the cases when the step duration request is frequently updated by the patient. Another approach, advocated in this article, is to develop a fast trajectory replanning methodology to be used online. This is a more flexible approach, able to deal with various experimental conditions and patients' morphologies.
The various available methods reported in~\cite{c1,c2,c5,c7} have too heavy computational burdens, which discards them from high-frequency online implementation on-board Atalante.

To ensure rapidity and responsiveness in real-time, we rely on a simple surrogate model of the system. Following e.g.\ \cite{kajita2003}, we approximate the patient-robot dynamics with a Linear Inverted Pendulum (LIP) model. This model involves only two (vector) variables: the \cop and the \com. The model is particularly insightful from a stability perspective. It is usually considered (see e.g.\ \cite{wieber2018}) that as long as there exists a controlled trajectory s.t.:
\emph{i)} the LIP \cop remains inside the support polygon at all times,
\emph{ii)} the LIP terminal condition of the nominal trajectory is reached,
then the robot controllers successfully achieve the trajectory tracking and are able to start a new step after the current one. These conditions guarantee the long-term stability of the walk hence, the patient's safety. 

In this article, a trajectory satisfying these assumptions is said to be \emph{feasible}. Because the robot is essentially behaving like an inverted pendulum subjected to gravity and controlled with inputs that are heavily constrained, the existence of a feasible trajectory depends on the time interval it is defined on. Thus, not all the user's requests (taking the form of a desired duration) can be considered as valid. Existence of feasible trajectories can be assessed by solving a constrained optimal control problem (OCP) of the LIP model. If the time specified by the user yields a feasible solution, then the step can be achieved with this duration. If not, the duration should be adjusted. Mathematically,
the OCP to be resolved is 
an input-constrained OCP for Linear Time-Invariant (LTI) dynamics with state
specified at an unspecified terminal time.

This is a classical problem which can be solved using a non-linear programming (NLP) approach, for instance using direct collocation, e.g. \cite{hargraves-paris-87}. Significant progress has been made toward solving similar problems using either reduced~\cite{c1,c2,c4} or full~\cite{c5, c7} robot models, e.g. Using state-of-the-art numerical solvers, it is possible to solve our problem, having an unspecified terminal time, every 10\,ms with a good level of accuracy on custom embedded hardware. Despite being fast, this level of performance is considered insufficient for the Atalante rehabilitation use-case. The primary cause for criticism is the perceived lack of reactiveness to the patient's efforts. In this paper, we show how to speed up these computations by a factor of 10. We exploit a mathematical property of the constraints to recast the NLP as a cascade of two optimization problems, using a bisection on Quadratic Programs (QP).

The main outcome of the article is a bisection algorithm on QP feasibility functions, which takes as argument the desired duration of the step requested by the patient and outputs the optimal feasible velocity granting safe execution of the step. The CoM trajectory is a readily obtained by-product of the proposed algorithm. It is grounded on a formal result describing the feasibility of the LIP trajectories. An optimization algorithm has already been proposed in~\cite{c8} for solving bi-level problems with quadratic lower levels, but without any guarantee of finding the global optima. Here, in the case of the LIP, we provide the proof that the problem has at most one local optimum additionally to the global one, in theory. We numerically check that there is actually none in our 2D use-case. In addition, we provide a simpler algorithm to find the optimal solution, using bisection, which leverages the 1D nature of our higher-level objective.

The paper is organized as follows. In~\cref{sec:sectionII}, we present the LIP model and the replanning OCP. The nature of the set of feasible trajectories is studied in~\cref{sec:phaseSpaceAnalysis}, and the main~\cref{mainT} is formulated. Using a detailed phase space analysis and Pontryagin Maximum Principle (PMP), we prove~\cref{mainT} in~\cref{sec:proofMainT}.
\cref{mainT} states that the set of feasible terminal times is either an interval of $\mathbb{R}^+$, or the reunion of two intervals of $\mathbb{R}^+$.
This result is instrumental in the design of our numerical resolution method.
In~\cref{sec:sectionV}, we propose a bisection method to solve cases of practical interest for physiotherapists. The numerical algorithm is tested on a scenario of a highly varying user demand taking the form of strongly varying velocity along the gait (with variations over $50\%$). The resolution time is below $1\textrm{ms}$, which stresses the responsiveness to the patient inputs.
Finally, in~\cref{sec:sectionVI}, this methodology is tested on a high-fidelity full-body simulator. Extensive numerical experiments serve to determine the performance of this novel approach and its limitations. They also stress the representativeness of the LIP model. Fall only occurs at extremely high velocities or prolonged periods of near-zero velocity. Thanks to our algorithm, the fall rate drops from $30\%$ (when using naive time scaling) to only $8\%$.
\section{LIP model and feasibility}\label{sec:sectionII}
{The Linear Inverted Pendulum (LIP) reduced model of the patient-robot system is presented below, along with a fixed-time OCP. It enables us to formulate a rehabilitation task as an optimization problem.}
\subsection{The LIP model: a reduced model of the exoskeleton}
The LIP  model is a low-dimensional  control model commonly considered in the robotics community.
The main assumptions necessary for its construction are briefly stated below, and follow~\cite{cWieber}. Consider the robot depicted in~\cref{fig:exo}, seen as a rigid body of mass $m$ on a horizontal ground. Newton's second law writes $m(\ddot{c}+g) = \sum_i f_i$
with $c\in \mathbb{R}^3$ the position of its \com (CoM), $g$ the gravity vector, and $f_i$ the contact forces.
Euler's equation, with respect to the CoM, writes $
   \dot{L} = \sum_i(p_i - c)\times f_i
$ 
with $L$ the angular momentum of the whole robot with respect to its CoM, $p_i$ the point of application of $f_i$, and $\times$ the cross-product of $\mathbb{R}^3$.
Then,
\begin{equation}\label{eq:flatGroundNewtonEuler}
   \frac{mc\times (\ddot{c}+g) + \dot{L}}{m(\ddot{c}^z + g^z)} = p
\end{equation}
with $p \triangleq \frac{\sum p_i^{x,y} f_i^z}{\sum f_i^z}$ the \cop (CoP). Following the admittance paradigm~\cite{stairClimbing}, this variable can be controlled, and we note it $u=p$ from now on to designate it as an input to the controlled dynamics.
Assuming the CoM has no vertical motion ($\ddot{c}^z = 0$), and the angular momentum is constant ($\dot{L} = 0$),
\cref{eq:flatGroundNewtonEuler} simplifies into the \emph{LIP model}
\begin{equation}\label{LIPmodel}
   \ddot{c}^{x,y} = \omega^2(c^{x,y} - u^{x,y})
\end{equation}
with $\omega = \sqrt{\frac{g}{c^z}}$.
As $x$ and $y$ dynamics of the LIP model are \emph{decoupled}, $.^{x,y}$ notations will be omitted for the rest of this paper, and we consider the single dimensional second-order dynamics
\begin{equation}\label{eq:LIPdyn}
   \ddot{c} = \omega^2(c-u)\ \text{ with }c\in\mathbb{R}.
\end{equation}

We perform the following change of coordinates to diagonalize the state equations above with
$x_1 = \xi \triangleq c + \frac{\dot{c}}{\omega},\ x_2 = \xi - 2c = -c + \frac{\dot{c}}{\omega}$, where $\xi$ denotes the Divergent Component of Motion (DCM). Then, \cref{eq:LIPdyn} takes the diagonal form $\dot{x} \triangleq Ax + Bu$ with
\begin{equation}\label{eq:x1x2Dyn}
   \begin{cases}
      \dot{x}_1  = \omega(x_1 - u)\\
      \dot{x}_2  = \omega(-x_2 - u)
   \end{cases}
\end{equation}
The solution of~\cref{eq:x1x2Dyn} with input $u$, from the initial condition $x^0\in\mathbb{R}^2$ is denoted $x^u$.
\subsection{Optimal control problem (OCP) and feasibility criterion}
Below, we propose a feasibility criterion for step durations as the existence of a solution to an OCP.
By definition, the CoP belongs to the convex hull of all the contact points, also called the Support polygon $Sp$.
Therefore, in~\cref{eq:LIPdyn}, $u$ belongs to $Sp$.

In the following$\footnote{In principle, $Sp$ depends on future decision variables and changes as the contact changes.}$, we restrict ourselves to a fixed set $Sp$. Further, it is assumed to be of rectangular form, so that $u\in\mU \triangleq [u_m, u_M]$ with $u_m < u_M$. This allows to cover scenarios of replanning until the end of the current step. By definition, $Sp$ corresponds to the support foot.

To guarantee the long-term stability of the walk, a punctual final constraint is introduced $x(T) = x^f\in\mathbb{R}^2$, with $T>0$ the optimization horizon.

\begin{defn} Consider the set of admissible controls
$
   U_{ad}(T) \triangleq \left\{u\  \text{s.t.}\ \forall t\in[0,T],\ u(t) \in \mU\right\}
$.
A duration $T$ is \emph{feasible} if 
\begin{equation*}
   \Omega(x^0,x^f,T) \triangleq \left\{u \in U_{ad}(T),\ x^u(0) = x^0,\ x^u(T) = x^f\right\}
\end{equation*}
is not empty.
\end{defn}
We denote $\mT(x_0,x_f)$, or $\mT$ for brevity purposes, the set of feasible times $\mT (x_0,x_f) \triangleq \{T> 0,\ \Omega(x^0,x^f,T)\neq\emptyset\}$.
\subsection{Rehabilitation and problem statement}
During rehabilitation, the patient specifies a desired step duration\footnote{The actual process by which the patient specifies this parameter is out-of-the-scope of the paper. We refer the interested reader to~\cite{JolyVG}.} $T^t$.
We propose to solve the following cascaded optimization problems\footnote{The presented quadratic cost function can be easily changed to incorporate extra tuning parameters to
affect performance, without loss of generality.}, which aims at satisfying this request while ensuring safety.
\begin{problem}[Replanning over an unspecified horizon]
   Given $x^0$, $x^f$ and $T^t$, find $T^*$ and $u^*$ as
   \begin{align}
         &T^* = \argmin_{T\in\mT(x^0,x^f)} |T-T^t|\label{pbT}\\ 
         &u^* = \argmin_{u\in\Omega(x^0,x^f,T^*)}\int_0^{T^*} u^2 dt\label{pbu}
   \end{align}
\end{problem}

For a given $T$ s.t.~$\Omega(x^0, x^f, T)\neq \emptyset$, determining $u^*$ in~\cref{pbu} is a fixed horizon input constrained LTI problem, which can be readily solved numerically because it is convex. A more challenging point is the description of the set $\mT$ constraining~\cref{pbT}. It is the subject of the following section where we perform an analysis of the trajectories of~\cref{eq:LIPdyn} in the phase plane to characterize the nature of $\mT$, and derive our main result.
\section{Phase space analysis and main result}\label{sec:phaseSpaceAnalysis}
Below, we study the solutions of minimal and maximal time OCPs. This study stresses the role of several regions in the phase plane being key in the reachability of a target $x^f$ from an initial condition $x^0$.

Then, we state our main result~\cref{mainT}. Its proof is provided in the next section.

\subsection{Definitions}\label{def}
A piecewise constant control input $u$ having $N$ steps over an interval $[0,t_f]$ is defined using a finite (irreducible) partition $0<d_1<d_2+d_1<...<d_n+...+d_1=t_f$ with $u$ taking values only in $\left\{ u_m,u_M \right \}$. For convenience, it is described by its first value and the durations, e.g.\ for 3 steps of respective durations $d_1$, $d_2$, $d_3$ starting with $u_m$, a sequence
$
   (u_m, d_1,d_2,d_3) =\seq\mapsto u
$ 
gives $u(t)=u_m$ for $0\leq t <d_1$, $u(t)=u_M$ for $d_1\leq t <d_2+d_1$, $u(t)=u_m$ for $d_2+d_1\leq t <d_3+d_2+d_1$.
For any initial condition $x^0$, and any $\seq$ defining a control $u$ as detailed above, over $\tau \in [0,t_f]$ we note the solution $x^{\seq}\triangleq x^u$ of the differential equation $\dot x=Ax+Bu$ which is
\begin{equation*}
   \phi(x^0, \seq, \tau) \triangleq x^{\seq}(\tau) = e^{A\tau}x^0 + \int_0^\tau e^{A(\tau - s)}Bu(s)ds
\end{equation*}
By extension, we define
$
   \phi(x^f, \seq, -\tau) \triangleq e^{-A\tau}x^f - \int_0^\tau e^{-A(\tau - s)}Bu(T - s)ds
$.

We define several subsets of $\mathbb{R}^2$ as follows
$\mD \triangleq \{\XY,\ x_2=-x_1\}$, $\mD^+ \triangleq \{\XY,\ x_2 > -x_1\}$, $\mD^- \triangleq \{\XY,\ x_2 < -x_1\}$, and 
$\mU_m^- \triangleq \{\XY,\ x_1<u_m\}$, $\mU_m^+ \triangleq  \{\XY,\ x_1>u_m\}$, 
with the same notations for $\mU_M$.
Finally, we define two open double cones
$\mC_M \triangleq  \{\mD^+\cap\mU_M^-\}\cup\{\mD^-\cap\mU_M^+\}$, $\mC_m \triangleq  \{\mD^-\cap\mU_m^+\}\cup\{\mD^+\cap\mU_m^-\}$.

Zero-order hold of $u$ for a duration $d= t_2 - t_1$ yields the solution
\begin{equation}\label{eq:x1x2Sol}
   x^u(t_2) = \begin{pmatrix}
      e^{\omega d} & 0\\ 0&e^{-\omega d}
   \end{pmatrix}x^u(t_1) + \begin{pmatrix}
      1 - e^{\omega d}\\e^{-\omega d} - 1
   \end{pmatrix}u
\end{equation}
For all vectors variables a subscript $_1$ or $_2$ indicates the first or second coordinate.
\subsection{Preliminary results on optimal trajectories and phase portrait}\label{prelimResults}
\begin{lemma}\label{lem:minMaxSolGlobalAndExistance}
   For all $(x^0,x^f)\in\mathbb{R}^4$ and $T>0$, if there exists a solution $u\in \Omega(x^0,x^f,T)$, then a minimum time solution (noted $\um$) always exists and, when the set $\mT$ is upper-bounded, a maximum time solution (noted $\uM$) exists. They are global optima.
\end{lemma}
\begin{proof}
   \cref{eq:x1x2Dyn} is linear, and $U_{ad}$ is compact and convex, hence, when a solution $u\in \Omega(x^0,x^f,T)$ exists, a minimum time solution $\um$ exists from~\cite[Theorem~4.3]{liberzon}.

   When the set $\mT$ is upper-bounded, we note $\overline{T}$ its supremum. Given a sequence $(T_k, u_k)$ s.t. $\lim_{k\rightarrow \infty} T_k = \overline{T}$, consider the sequence $(\overline{T}, \tilde{u}_k)$ of prolonged $u_k$ on $[T_k, \overline{T}]$ by the null function, then the proof provided in~\cite[Theorem~4.3]{liberzon} is straightforwardly extended to the $\tilde{u}_k$ sequence, yielding the existence of $\uM$. Hence, $\overline{T}$ is maximum.
\end{proof}

We denote \begin{equation}\label{eq:minTime}
 \begin{aligned}(T_{\min}, \um) \triangleq &\argmin T,\\
    u\in \Omega(x^0&,x^f,T>0)
 \end{aligned}\quad
 \begin{aligned}(T_{\max}, \uM) \triangleq &\argmin -T\\
    u\in \Omega(x^0&,x^f,T>0)
 \end{aligned}
\end{equation}
\subsection{Main result}\label{mainResult}
\begin{theorem}[Description of \texorpdfstring{$\mT$}{ T}]\label{mainT}
The set of feasible times $\mT$ is either empty, or of the form $[\Tm,\TM]$, or of the form $[\Tm,+\infty[$, or of the form $[\Tm,A]\cup[B,+\infty[$, $A<B$.
\end{theorem}

\cref{mainT} is instrumental for numerically solving~\cref{pbT}. Knowing that $\mT$ is composed of one or two intervals, the solution is simply the projection of  $T^t$ onto them. As detailed in \cref{sec:sectionV}, $\mT$ is composed of a single interval in our practical case, therefore the projection is readily obtained by a bisection method applied to the feasibility function of a quadratic program.
\section{Proof of~\texorpdfstring{\cref{mainT}}{ Theorem 1}}\label{sec:proofMainT}
Below, we first exhibit in~\cref{sec:proofPreliminaries} particular regions of the phase portrait which serve to organize the proof. 
We study the boundedness of~$\mT$ in~\cref{sec:boundedness}, then assess its convexity properties in~\cref{sec:convexBounded} and~\cref{sec:convexUnBounded}.
\subsection{Regions of interest in the phase portrait} \label{sec:proofPreliminaries}
\begin{lemma}\label{lem:minMaxTimeSol}
   The solution $\um$ is bang-bang, i.e.\ takes only values in $\{u_m, u_M\}$,
   with a maximum number of one switch.
   The same property holds for $\uM$ when it exists.
\end{lemma}

\begin{proof}
   Consider the Hamiltonian $
      H(t, \lambda^0, \lambda, x, u) = \mu + \lambda(t)(Ax + Bu)
   $.
   Using the PMP, the adjoint equation and solution write
   $
      \dot{\lambda} = -\frac{\partial H}{\partial x} = - \lambda(t)A,\quad
      \lambda(t) = \lambda^0e^{-At} 
  $
  and the switching function is $
      \Gamma(t) = \lambda B = \lambda^0{e^{-At}}B = -\omega\lambda^0\begin{pmatrix}e^{-\omega t} &e^{\omega t}\end{pmatrix}^T
$.
   If ${\lambda^0}_1 {\lambda^0}_2 < 0$, then there exists a unique switching time $\frac{1}{2\omega}\log(-\frac{\lambda^0_1}{\lambda^0_2})$ for which $\Gamma$ changes sign.
   Otherwise, $\Gamma$ has a constant sign. This concludes the proof for $\um$.
   The proof regarding $\uM$ is identical.
   
\end{proof}
\begin{figure}[thpb]
   \centering
   \includegraphics[scale=0.28]{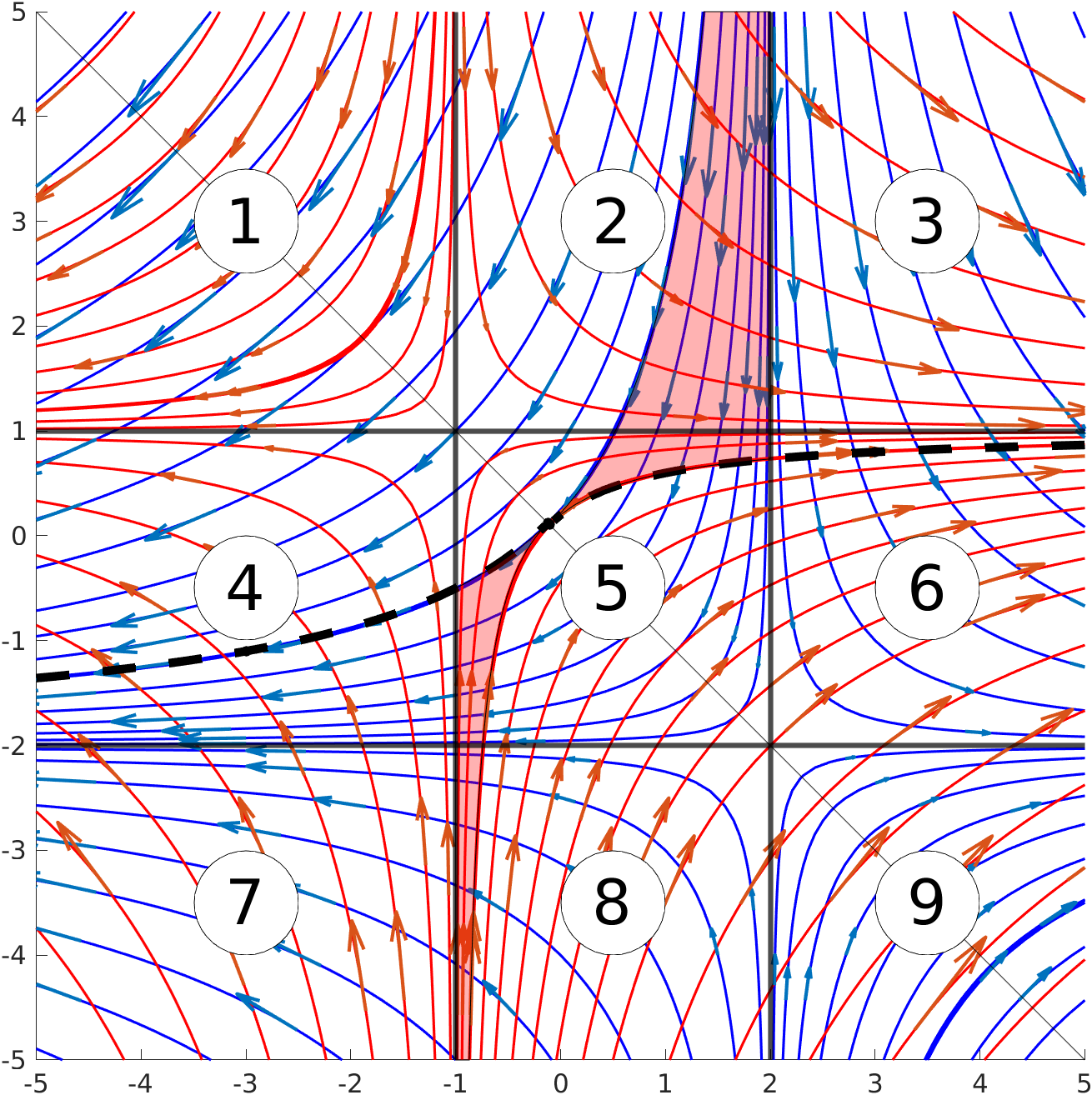}
   \caption{Phase diagram for~\cref{eq:x1x2Dyn} with $u_m=-1$ (red) and $u_M=2$ (blue). $\mC_m$ covers $\region{1}\cap\mD^+$, $\region{5}\cap\mD^-$,  $\region{9}\cap\mD^-$ and  $\region{8}$.  $\mC_M$ covers $\region{1}\cap\mD^+$, $\region{5}\cap\mD^+$,  $\region{9}\cap\mD^-$ and  $\region{2}$. }
   \label{fig:stateSpaceSplit}
\end{figure}
\cref{lem:minMaxTimeSol} highlights the importance of the phase portrait in~\cref{fig:stateSpaceSplit} corresponding to constant control values $u_m$ and $u_M$. It is split into nine open regions, some of them being open semi-infinite strips,  whose boundaries are the trajectories passing through the equilibrium points for $u_m$ and $u_M$.
We denote each region $\region{i},\ i = 1,...,9$. Also, we will note $\region{ijk...} \triangleq \region{i}\cup\region{j}\cup\region{k}\cup...$ for any number of indexes.
Notice two interesting properties: \emph{i)} the locus of intersecting parallel arcs is $\mD$ and \emph{ii)}
the subsets $\region{147}$ and $\region{369}$ are positively invariant under the controlled flow.

Next, the following result states that in the cone $\mC_m$ (resp. $\mC_M$), the flow corresponding to $u_m$ (resp. $u_M)$ reaches a point symmetric to the initial condition with respect to the line $\mD$. This property is instrumental in the proof.
\begin{lemma}\label{lem:doubleCone}
   For all $x$ in the double cones $\mC_m\cup\mC_M$, we have $\phi(x,u,f(x, u)) = Sx$  with $f(x, u) \triangleq \frac{1}{\omega}\log(\frac{u+x_2}{u-x_1})$, $S=\begin{pmatrix}0&-1\\-1&0\end{pmatrix}$, $u = u_m$ if $x\in\mC_m$, and $u = u_M$ otherwise.
\end{lemma}

\begin{proof}
In the double cones $\mC_m$ and $\mC_M$,
$f$ is well-defined as a function of its arguments. A direct calculus with~\cref{eq:x1x2Sol} yields the conclusion.
\end{proof}

\subsection{Boundedness of \texorpdfstring{$\mT$}{T}} \label{sec:boundedness}
Depending on the values of $x^0$ and $x^f$, the set $\mT$ can be empty ($\emptyset$), bounded ($\bounded$), or unbounded ($\infty$). 
\begin{lemma}[Boundedness of \texorpdfstring{$\mT$}{T}]\label{lem:lemmaBoundedness}
   Conditions on $x^0$ and $x^f$ corresponding to cases of non-empty $\mT$ are listed in~\cref{fig:x0xfSplit}.
\end{lemma}
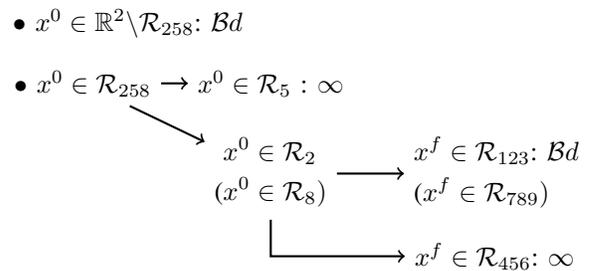
\begin{figure}[thpb]
   \begin{tikzpicture}[node distance={20mm}, thick]
      \node (2) {$\bullet\ x^0 \in \region{258}$}; 
      \node (3) [above of=2, xshift=0.6cm, yshift=-1.2cm] {$\bullet\ x^0 \in \mathbb{R}^2\backslash\region{258}$: $\bounded$};
      \node (5) [right of=2, xshift=0.5cm] {$x^0 \in \region{5}$
      : $\infty$
      }; 
      \node (4) [below of=5, yshift=0.8cm] {$\begin{aligned}&x^0 \in \region{2}\\ \text{(}&x^0 \in \region{8})\end{aligned}$}; 
      \node (8) [right of=4, yshift=0.cm, xshift=1cm] {$\begin{aligned}&x^f \in \region{123}\text{: }\bounded\\ &\text{(}x^f \in \region{789})\end{aligned}$};
      \node (9) [below of=8, yshift=0.9cm, xshift=-0.025cm] {$x^f \in \region{456}\text{: }\infty$}; 
      \draw[->] (2) -- (4);
      \draw[->] (2) -- (5);
      \draw[->] (4) -- (8);
      \draw[->] (4) |- (9);
   \end{tikzpicture} 
   \centering
   \caption{Graph of all possible cases  of non-empty $\mT$.}
   \label{fig:x0xfSplit}
\end{figure}
\begin{proof}
We split the proof according to the location of $x^0$ in the phase plane and, when necessary, the location of  $x^f$. Only cases corresponding to non-empty $\mT$ are considered.


For $x^0 \in \mathbb{R}^2\backslash \{\region{258}\}$, the argument stems from the monotonic divergence of $x_1$. For instance $x^0\in\region{147}$, there exists $\epsilon > 0$, s.t. $x_1^u(0) \leq u_m - \epsilon$. Then, using~\cref{eq:x1x2Sol}, one easily shows that
$\forall t>t_0,\ \dot{x}_1 = \omega(x_1 - u) \leq -\omega\epsilon$.
Therefore, the final time is upper bounded by $\frac{x_1(t_0) - {x^f}_1}{\omega\epsilon}$.
A similar inequality is obtained for $x_1^u(t_0) \geq u_M + \epsilon$ to cover $\region{369}$.
Hence, $\mT$ is upper bounded.


For $x^0\in\region{5}\backslash\mD$, which is entirely covered by $\mC_m\cup\mC_M$, and is stable by symmetry with respect to $\mD$.
\Cref{lem:doubleCone} permits to build a sequence that periodically returns to $x^0$, prolonging infinitely any solution from $x^0$. Hence, $\mT$ is not upper-bounded.

For $x^0\in\region{5}\cap\mD$, for all possible values of $u$, the tangent vector field at $x^0$ is orthogonal to $\mD$ and does not vanish. For any short time the preceding rationale applies.

For $x^0\in\region{2}$ and $x^f \in \region{123}$, one has $\dot{x}_2 < 0$, therefore $x_2$ is decreasing, hence $\dot{x}_2 \leq -\omega({x^f}_2 + u_m)$. Therefore,  $\mT$ is upper-bounded by $ \frac{{x^0}_2 - {x^f}_2}{\omega({x^f}_2 + u_m)}< \infty$.

For $x^0\in\region{2}$ and  $x^f \in \region{456}$,
there exists a sequence from any $x^0$ s.t., for some $t>0$, $x_w \triangleq \phi(x^0,\seq,t)\in \region{5}$. 
In addition, any $x^f$ can be accessed from this waypoint $x_w$ through a sequence $(u_m, a,b)$ or $(u_M, a,b)$, with $a,b>0$.
Therefore, a transient from $x^0$ to $x^f$ passing through $x_w$ can be arbitrarily prolonged with sequences periodically returning to $x_w$.
Hence, $\mT$ is not upper-bounded.

The case $x^0\in\region{8}$ the analysis is similar to $x^0\in\region{2}$.

This completes the proof.
\end{proof}
\subsection{Convexity of bounded \texorpdfstring{$\mT$}{T} cases}\label{sec:convexBounded}
\begin{lemma} \label{lem:lemmaConvexBounded}
When $\mT$ is bounded, $\mT = [\Tm, \TM]$.
\end{lemma}

\begin{proof}
\Cref{lem:lemmaBoundedness} shows that for $\mT$ to be bounded either $x^0 \in \mathbb{R}^2\backslash\region{258}$, or $(x^0,x^f) \in \region{2}\times\region{123}$, or $(x^0,x^f) \in \region{8}\times\region{789}$.

As we only consider the case of bounded $\mT$ in this section, \cref{lem:minMaxSolGlobalAndExistance} shows the existence of solutions of~\cref{eq:minTime}. In general,
there are at most two bang-bang sequences with one switch between $x^0$ and $x^f$ which are noted $\seq_{m}\triangleq(u_m,a_m,b_M)$ and $\seq_{M}\triangleq(u_M,a_M,b_m)$.
Further, according to~\cref{lem:minMaxTimeSol}, the two controls $\um$ and $\uM$ are bang-bang with at most one switch. Hence, either $\um = \seq_m$ and $\uM = \seq_M$, or the other way around.

By definition, $\mT\subset [\Tm, \TM]$. When $\Tm=\TM$, $\mT$ is a singleton, hence is trivially convex. We now assume $T_{\min}<T_{\max}$.
The rest of the proof depends on the location of $(x^0, x^f)$ relative to $\mD$.
\begin{figure}[thpb]
   \centering
   \includegraphics[width=.8\linewidth]{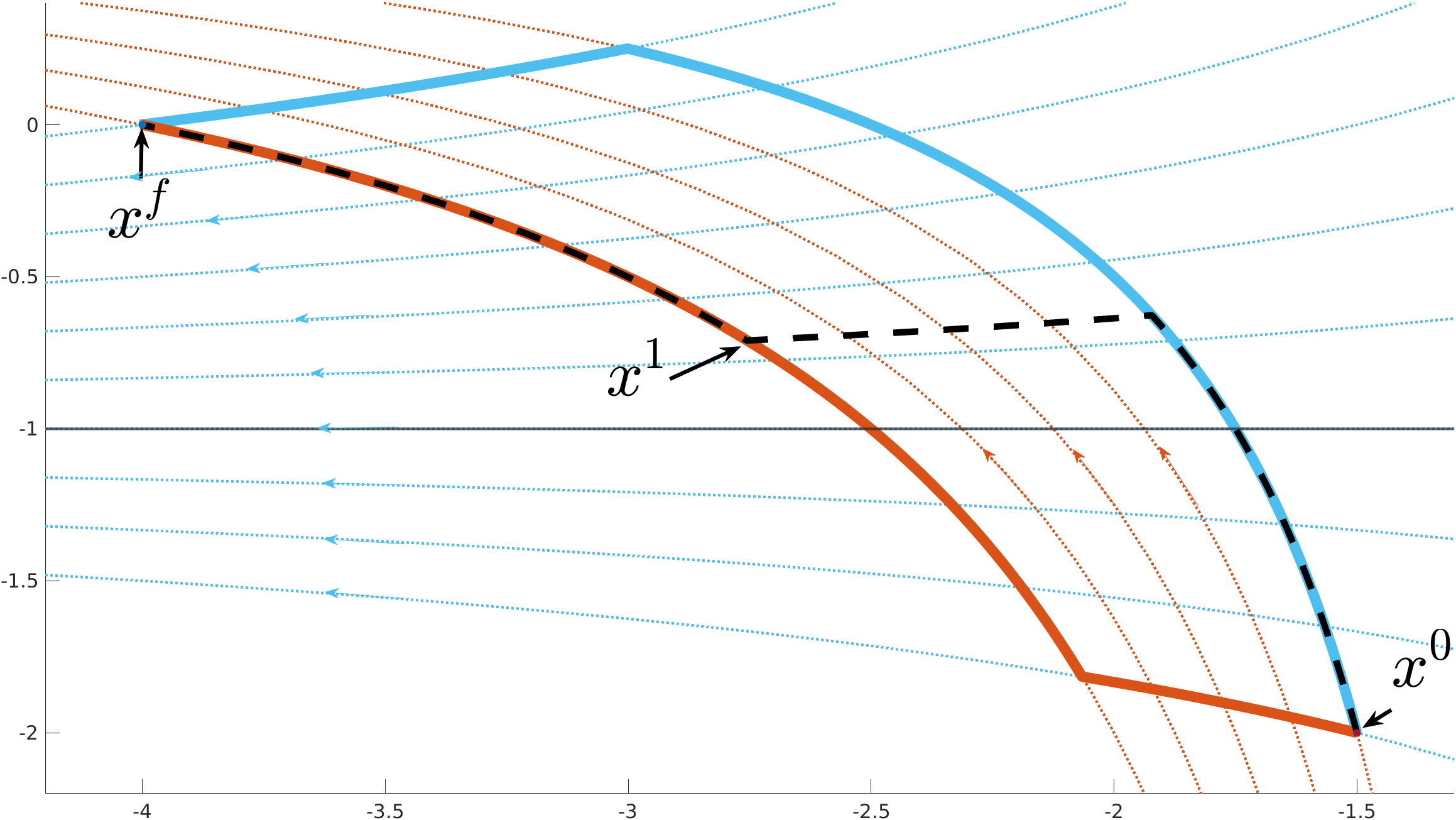}
   \caption{Phase diagram for~\cref{eq:x1x2Dyn} with $u=u_m$ (dotted red), $u=u_M$ (dotted blue),
   minimum time trajectory $x^{\min}$ (solid blue line), maximum time trajectory $x^{\max}$ (solid red line),
   and $x^{\seq_v}$ for some $a\in[0,a_m]$ (dashed black line).}
   \label{fig:figA4A7}
\end{figure}
\subsubsection{\texorpdfstring{$(x^0, x^f) \in{\mD^-}^2$}{(x0,xf) in D-} (on the same side)}\label{caseD--}
\cref{fig:figA4A7} shows the trajectory steering the system from $x^0$ to $x^f$ using $\seq_m$ and $\seq_M$.
The trajectories do not cross each other. The trajectory corresponding to $\seq_m$ and the reverse trajectory corresponding to $\seq_M$ define a positively oriented closed curve. 
The region $\region{147}\cap\mD^-$ being positively invariant, the curve is strictly included in ${\mD^-}^2$.

For all $x\in\mD^-$, 
\begin{equation*}
      \left<\begin{pmatrix}0 & -1\\1 &0\end{pmatrix}(Ax+Bu_m),\ Ax+Bu_M\right>\ > 0
\end{equation*}
Therefore, for all $a\in[0, a_m]$, there exists $b\geq 0,\ c\geq 0$ s.t.
\begin{equation*}
   \seq_v (a,b,c) \triangleq (u_m,a,b,c) \in \Omega(x^0, x^f, a+b+c)
\end{equation*}
The solution $x^{\seq_v}$ is shown in~\cref{fig:figA4A7}.
By definition, the solution $x^{\seq_v}$ is continuous, hence
\begin{equation*}
   x_1 \triangleq \phi(x, \seq_v (a,b,c), a+b) = \phi(x^f, \seq_v (a,b,c), -c)
\end{equation*}

We define the function $
   \fT :  
      [0,a_m] \ni a \mapsto \fT(a) = a+b+c \in [T_{\min}, T_{\max}]
       $
which maps the duration $a$ of the first arc of $x^{\seq_v}$ to the total duration $a+b+c$. Define $g$ as
$
   g(a,b,c) = \phi(x, \seq_v(a,b,c), a+b) - \phi(x^f, \seq_v(a,b,c), -c)
$. From  $u_m\neq u_M$, one has
\begin{equation*}
   \begin{aligned}
      \rank(&\begin{bmatrix}
         \frac{\partial g}{\partial b}, \frac{\partial g}{\partial c}
      \end{bmatrix})
      = \rank(\begin{bmatrix}
         Ax^1 + Bu_M & -Ax^1-Bu_m
      \end{bmatrix})   \\
     & \phantom{eeeee}= \rank(\begin{bmatrix}
         Ax^1& B
      \end{bmatrix})
       = \rank(\begin{pmatrix}
         x^1_1 & 1\\
         -x^1_2 & 1
      \end{pmatrix}) = 2
   \end{aligned}
\end{equation*}\label{fullrank}
The intermediate point $x^1$ is defined by $g(a,b,c) = 0$. The full rank property above associated to the injectivity of the function $(a,b,c)\mapsto\begin{pmatrix}
   a & g(a,b,c)
\end{pmatrix}^T$ gives, through the global inversion theorem~\cite[Theorem 6.2.3]{implicitFTH}, the existence of $\psi\in\mC^0$ s.t., over the domain of definition $[0,a_{\min}]$, $(b,c) = \psi(a)$.

Thus, the function $\fT$ is continuous. Therefore, by the intermediate value theorem, $[\Tm, \TM]\subset \fT([0, a_m]) \subset \mT$, which concludes the proof.

\subsubsection{\texorpdfstring{$(x^0, x^f) \in{\mD^+}^2$}{(x0,xf) in D+} (on the same side)}
The proof is identical, replacing $\mD^-$ by $\mD^+$, the trajectory corresponding to $\seq_m$ and the reverse trajectory corresponding to $\seq_M$ defining a negatively oriented closed curve.
\begin{figure}[thpb]
   \centering
   \includegraphics[width=0.72\linewidth]{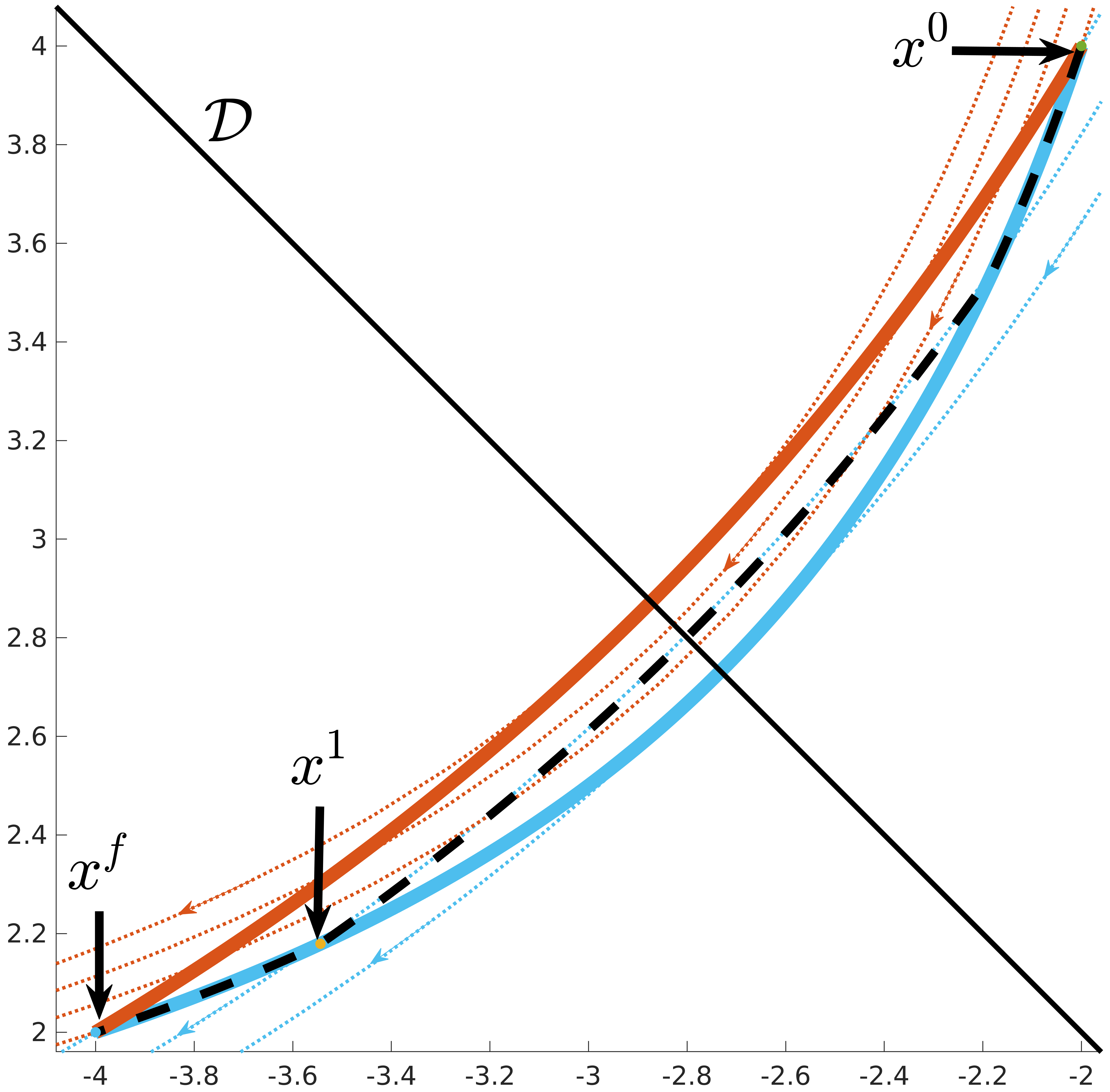}
   \caption{Phase diagram for~\cref{eq:x1x2Dyn} with $u=u_m$ (dotted red), $u=u_M$ (dotted blue),
   minimum time trajectory $x^{\min}$ (solid red line), maximum time trajectory $x^{\max}$ (solid blue line),
   and $x^{\seq_\mathcal{B}}$ for some $a\in[0,\frac{t_0}{2}]$ (dashed black line).}
   \label{figA1}
\end{figure}
\subsubsection{\texorpdfstring{$(x^0, x^f) \in \mD^+\times\mD^-$}{(x0,xf) in D+XD-} (on opposite sides)}\label{caseD+D-}
According to~\cref{fig:x0xfSplit}, $x^0\in\region{1}\cap\mD^+$ and $x^f\in\left\{\region{1}\cap\mD^-\right\}\cup\region{4}$.

If the (Euclidean) distance $d(x^f,\mD)$ between $x^f$ and $\mD$ is strictly lower than the distance $d(x^0,\mD)$ between $x^0$ and $\mD$, then~\cref{lem:doubleCone} states that $\phi(x^f,u_M,-t_0) = Sx^f = \phi(x^f,u_m,-t_1)$, for some $t_0,t_1 \geq 0$. We use the same constructive proof between $Sx^f$ and $x^f$ with the sequence $\seq_\mathcal{B} \triangleq (u_m, a, b, c)$, with $a\in[0,\frac{t_0}{2}]$.

If $d(x^f,\mD) > d(x^0,\mD)$, then~\cref{lem:doubleCone} states that $\phi(x^0,u_M,t_0) = Sx^0 = \phi(x^0,u_m,t_1)$, for some $t_0,t_1 \geq 0$. We use the same constructive proof between $Sx^0$ and $x^0$ with $\seq_\mathcal{B}$.

If $d(x^f,\mD) = d(x^0,\mD)$, the proof directly follows from $x^0=Sx^f$ to $x^f$, this situation is illustrated in~\cref{figA1}.
\subsubsection{\texorpdfstring{$(x^0, x^f) \in \mD^-\times\mD^+$}{(x0,xf) in D-XD+}  (on opposite sides)}
The case is identical to the previous case, $x^0$ belonging to $\region{9}\cap\mD^-$ and $x^f$ belonging to $\left\{\region{9}\cap\mD^+\right\}\cup\region{6}$.

This completes the proof.

\end{proof}
\subsection{Convexity of unbounded \texorpdfstring{$\mT$}{T} cases}\label{sec:convexUnBounded}
Let us define $\mJ$ a subset of $\mathbb{R}^4$ as follows
\begin{equation*}
   \mJ \triangleq  \left\{
      \begin{aligned}
      &x^0\in\region{258},\ x^f\in\region{456}\text{ \text{s.t.}} \\
      &\exists x_\mD, x_d \in \mD\cap\region{5},\ t_\mD> 0 ,\ t_d > 0, \\
      &\left\{\begin{aligned}
         &{x_d}_1 < {x_\mD}_1\\
         &x^f=\phi(x_\mD, u_m, t_\mD)\\
         &x^0=\phi(x_d, u_M, -t_d)
      \end{aligned}\right.
      \text{ or }
      \left\{\begin{aligned}
         &{x_d}_1 > {x_\mD}_1\\
         &x^f=\phi(x_\mD, u_M, t_\mD)\\
         &x^0=\phi(x_d, u_m, -t_d)
      \end{aligned}\right.
   \end{aligned}
    \right.
\end{equation*}
The set $\mJ$ is partially pictured in~\cref{fig:stateSpaceSplit} (all possible values of $x^0$ are colored in red when $x^f$ varies along the dashed line). It plays a particular role in~\cref{lem:lemmaConvexUnBounded} as it is the only one where boundary conditions yield a non-convex set $\mT$.

\begin{lemma} \label{lem:lemmaConvexUnBounded}
When $\mT$ is unbounded,
if $(x^0,x^f)\notin\mJ$,  then $\mT = [\Tm, +\infty[$, otherwise, there exists $A<B$ s.t. $\mT = [\Tm, A]\cup[B, +\infty[$.
\end{lemma}

\begin{proof}
Following~\cref{lem:lemmaBoundedness}, a careful investigation of the graph in~\cref{fig:x0xfSplit} reveals that for $\mT$ to be unbounded we have $(x^0, x^f)\in \region{258}\times\region{456}$. By symmetry of the vector field (rotation of $\pi$ about $(\frac{u_m+u_M}{2}, -\frac{u_m+u_M}{2})^T$), we now only consider a pair $(x^0,x^f)\in \left(\region{258}\cap\mD^+\right)\times\region{456}$.

\subsubsection{\texorpdfstring{$(x^0, x^f) \notin \mJ$}{(x0, xf) not in J}}
In all such cases, there exists a sequence $(u_m, a,b,c,d)$, with $a\geq 0, b>0, c>0, d\geq 0$ steering $x^0$ to $x^f$ with a single intersection with $\mD\cap\region{5}$. This sequence can be easily extended in the vicinity of $\mD\cap\region{5}$ (which excludes equilibria) to increase the transient time by any desired arbitrarily small increment $\epsilon>0$. Iteratively, this construction allows to infinitely increase the transient time by a continuous constructive process.

Also, the same type of sequence with other values for $a,b,c,d$ can generate a smooth collection of trajectories approaching the minimum time trajectory. The proof of~\cref{caseD--} yields the conclusion with the continuous mapping $\fT: (a,b,c,d)\mapsto a+b+c+d$.

\subsubsection{\texorpdfstring{$(x^0, x^f) \in \mJ$}{(x0, xf) in J}}
There exist two sequences $\seq_1 = (u_m,a_1,b_1)$ (e.g. corresponding to the minimum time $\Tm$) and $\seq_2=(u_M,a_2,b_2)$ (with time $T_2$) steering $x^0$ to $x^f$ by two paths $\Gamma_1$ and $\Gamma_2$ entirely in $\mD^+$. They are illustrated in~\cref{fig:holeFig}.

$\seq=(u_m,a,b,c)$ gives, by the continuity of $(a,b,c)\mapsto a+b+c$, that all feasible trajectories staying inside $\Gamma_1\cup\Gamma_2$ have a transient time in $[\Tm, T_2]$, for $T_2<\infty$.

Now, consider a trajectory from $x^0$ to $x^f$ leaving $\Gamma_1\cup\Gamma_2$. A detailed investigation of the phase portrait gives that this trajectory leaves $\Gamma_1\cup\Gamma_2$ at a point $x^{ii} \triangleq \phi(x^f, u_m, -t), 0<t\leq b_2$, strictly in $\region{5}$, with control $u>u_m$.

Hence, it exists
$ \epsilon > 0$ s.t. $x^{ii+} \triangleq \phi(x^{ii}, u>u_m, \epsilon)\in \mD^+\cap\region{5}.
$
From $x^{ii+}$ to $x^f$ the minimum time trajectory has a minimum time $T_{mini}$ and passes through $x^i\triangleq\phi(x^0, u_M, a_2)$, and 
\begin{equation*}
   \begin{aligned}
      &T_{mini}(x^{ii+}, x^f) = \\
      &\ T_{mini}(x^{ii+}, x^i) + T_{mini}(x^i, x^{ii}) + T_{mini}(x^{ii}, x^f)\\
      &\ \geq f(x^i, u_M) + f(Sx^i, u_m) + o(\epsilon) + T_{mini}(x^{ii}, x^f)
   \end{aligned}
\end{equation*}

By imposing $\epsilon \rightarrow 0$, we deduce that any such trajectory has a transient time larger than $T_3 = \min_{x^{ii}} T_{mini}(x^{ii},x^f) + T_0$, with $T_0 = f(x^i, u_M) + f(Sx^i, u_m)$. The transient time $T_3$ is given for a certain $x^{iii}\in\region{5}$. Hence, there exists a trajectory going through $x^0\rightarrow x^{iii} \rightarrow Sx^{iii} \rightarrow x^f$ with a transient time $T_3 + T_{mini}(x^0,x^{iii}) = T_4$, with a sequence $\seq=(u_m, a, b, c)$. By continuity, there exists $t \in [a, a+b]$ such that $\phi(x^0, \seq, t)\in\mD$.

This completes the proof.
\begin{figure}[thpb]
   \centering
   \includegraphics[width=.8\linewidth]{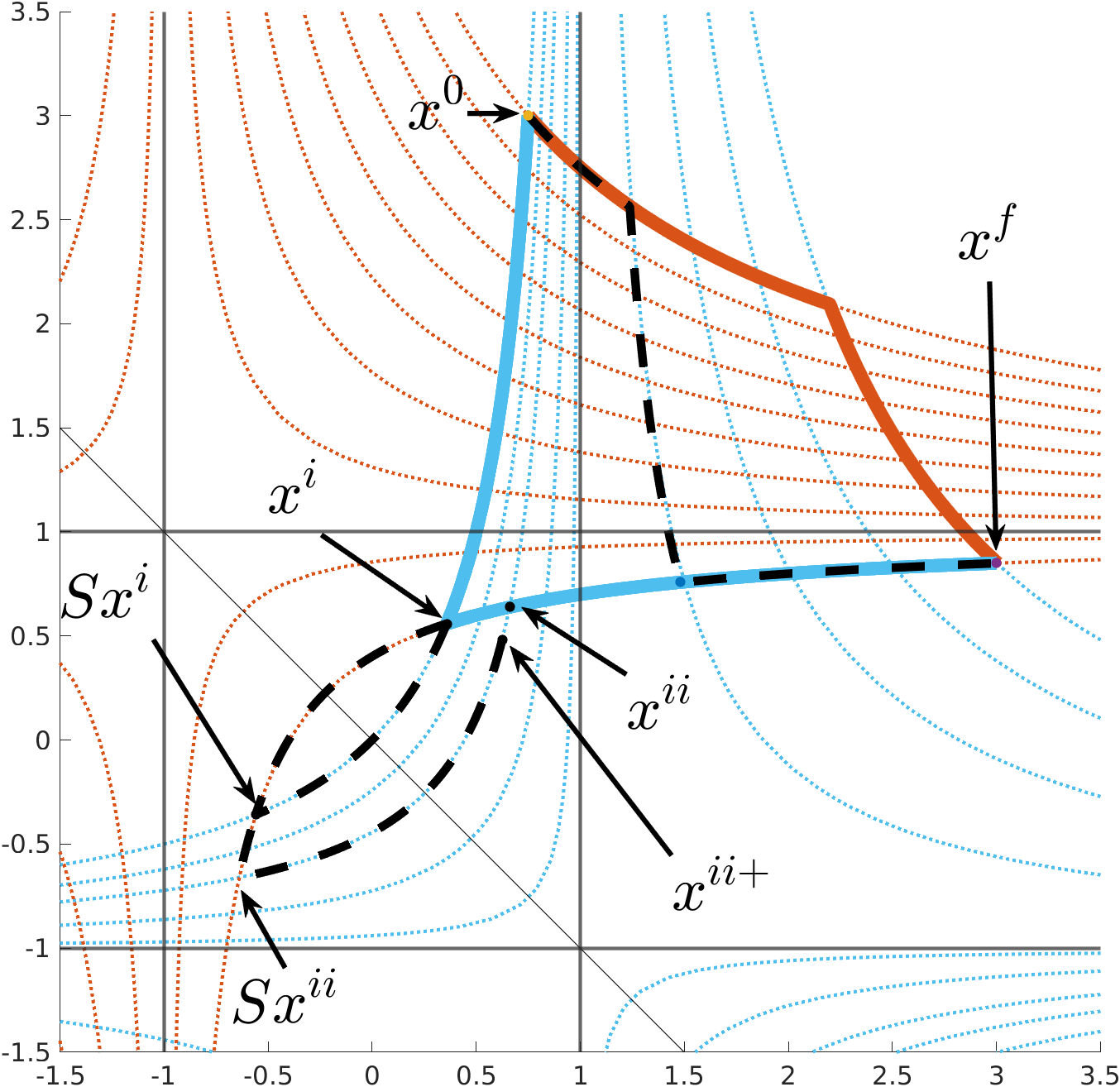}
   \caption{Phase diagram for~\cref{eq:x1x2Dyn} with $u=u_m$ (dotted red), $u=u_M$ (dotted blue),
   minimum time trajectory $x^{\min}$ (solid red line), $T_2$ time trajectory $x^{\seq_2}$ (solid blue line),
   and  points of interest for the proof of \cref{lem:lemmaConvexUnBounded}.}
   \label{fig:holeFig}
\end{figure}
\end{proof}
\section{Numerical method}\label{sec:sectionV}
\cref{mainT} describes $\mT$. In practice,~\cref{pbT,pbu} have to be considered in the two dimensions $x-y$ of~\cref{LIPmodel}. 
The two problems share a single parameter $T$. It has to belong to the two sets $\mT^x$ and $\mT^y$. This does not change the possible nature of $\mT=\mT^x\cap\mT^y$. We notice, numerically, that $\mT$ is a single interval, which enables us to use bisection to solve~\cref{pbT}. A side product is the resolution of~\cref{pbu}.

\subsection{QP resolution and feasibility check}
\Cref{pbu} defines a fixed-time OCP that can be addressed using a direct numerical method. Conveniently, the input signal is represented by a piece-wise $\mC^1$ function in between non-uniform nodes. The dynamics and the value of the integral cost are exactly represented using the first-order hold quadratures. This allows expressing boundary conditions and input constraints under an affine form in a finite number of variables, and the cost as a quadratic function of these variables. The same discretization procedure is employed in the $x-y$ directions, resulting in a QP with $2P$ variables and $4P + 4$ affine constraints. The outcome of the QP resolution is a feasibility boolean, and, when it is feasible, a solution to~\cref{pbu}.

\subsection{Bisection resolution on the feasibility}
We notice, numerically, that the nature of $\mT$ is a single interval. On~\cref{fig1}, each vertical slice of the $mT$ green area is a segment. This enables us to use bisection to solve~\cref{pbT}. 

Given an initial guess $T^0\in \mT$, we solve~\cref{pbT} using bisection on the feasibility function above  (treated as a boolean) between the target time $T^t$ and the initial guess $T^0$. 
Classically, the search interval is reduced by a factor $2^N$, where $N$ is the maximum number of iterations (typically 10).
Recursively, for the next time step, the guess is easily updated using the outcome of the previous run.

\section{Simulation results}\label{sec:sectionVI}
\subsection{Results for highly varying patient efforts}\label{optEvaluation}
A single-step gait is extracted from an available cyclic walk trajectory. To simulate the behavior of a highly demanding patient,  a strongly oscillating velocity along the geometric path is considered. The nominal velocity is $1$ and the variations are $\pm 50\%$. This defines a signal $t\mapsto T^t(t)$.
For reference, an exhaustive search algorithm is employed to determine at each step the feasible set $\mT$. As is visible in~\cref{fig1}, the resolution of~\cref{pbT} leaves the patient-chosen velocity unchanged at the beginning of the simulation. Gradually the feasible set gets more stringent and at some point, near $t=0.4$\,\textrm{s}, the proposed algorithm has to intervene. The desired time $T^t$ is no longer feasible on many occasions. The situation worsens until the end of the simulation. Notably, at the end, the walk has to be sped up significantly.
\begin{figure}[thpb]
   \centering
   \includegraphics[scale=0.16]{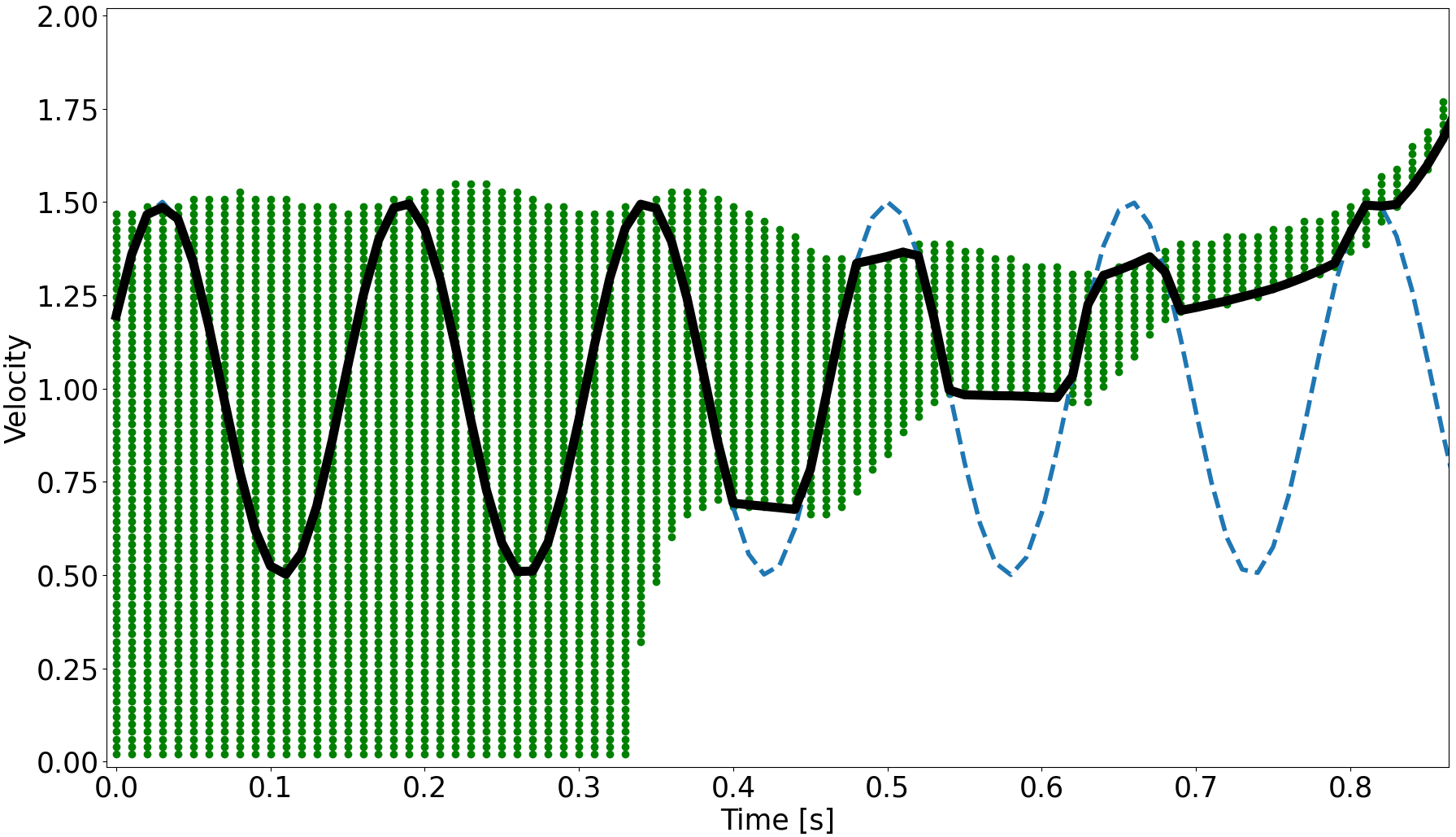}
   \caption{Velocity of the trajectory (1 is the nominal velocity). (Dotted blue): request from the patient. (Green dots): set $\mT$ determined by an exhaustive search, for reference. (Solid black): solution of the proposed methodology}
   \label{fig1}
\end{figure}
\subsection{Results on full-body simulations}
We perform extensive closed-loop rigid-body simulations of the patient-exoskeleton system to evaluate the safety increase offered by  our algorithm. To simulate the behavior of the patient, we consider piecewise velocity signals consisting of a square wave whose duration and magnitude are varied.    \cref{fig2} reports the results (for each duration magnitude, a vast list of possible starting times for the square disturbance is considered, and we report the success rate). A naive replanning methodology is used as a benchmark reference to illustrate the increased performance of our algorithm. It consists of a simple (and natural) time rescaling of the nominal articular trajectory using the simulated user velocity. \cref{fig2}~(left) reports simulation results obtained with the naive time rescaling methodology. \cref{fig2}~(right) reports the results obtained with our methodology.
A simulation is considered stable if the simulated patient-exoskeleton system walks for at least $10$\,\textrm{s} without falling.
In both cases, we use a state-of-the-art admittance-based DCM controller~\cite{stairClimbing} to stabilize around the reference CoM trajectory.
\begin{figure}[thpb]
   \centering
   \includegraphics[scale=0.23]{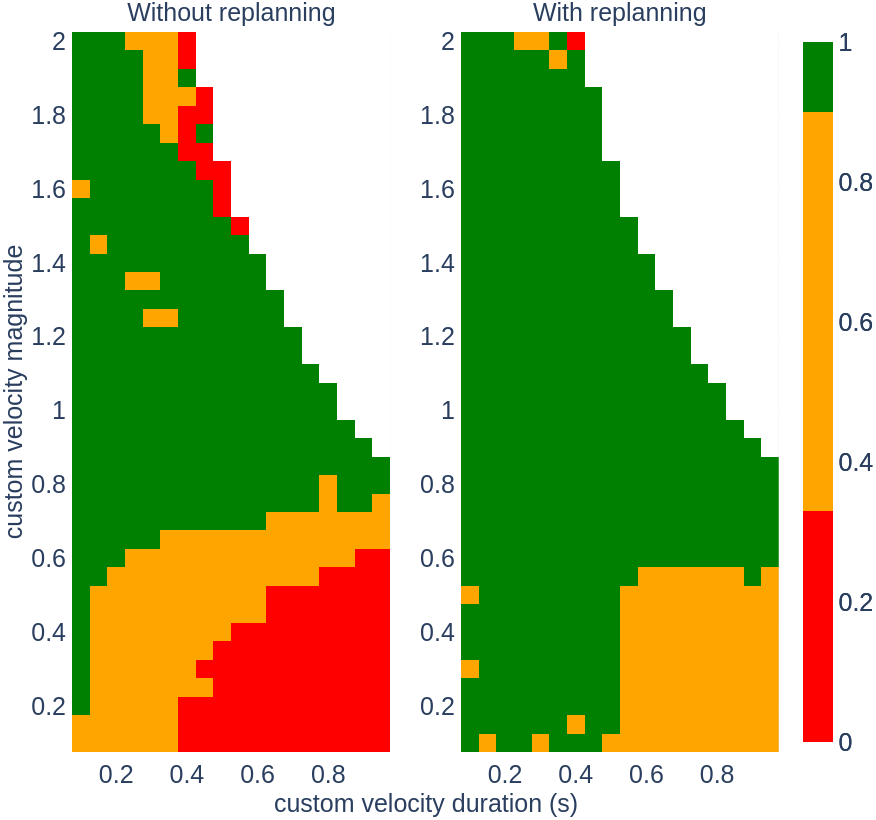}
   \caption[]{Comparison of rate of success heatmaps for velocity variations having various durations and magnitudes\protect\refstepcounter{footnote}\protect\footnotemark[\thefootnote]. Left: naive scaling. Right: with proposed replanning.}
   \label{fig2}
\end{figure}
\footnotetext[\thefootnote]{The white spaces in this figure corresponds to unfeasible values of the parameters violating the constraint that the scaled phase variable must remain smaller than 1.}

These results show the substantial improvement of the patient-exoskeleton system balance provided by the use of our algorithm, the stability being ensured for almost all considered cases, except for some very low-velocity cases with long durations (a careful examination of simulations reveals that fall occurs mostly when slow takes place at late stages of the step). A total of 2917 simulations have been conducted. In summary, less than $8\%$ of cases are failing our algorithm, while more than $30\%$ were without it.

\subsection{Computational load}
In view of applications, we will need to implement this algorithm at $1$\,\textrm{kHz}. Typical numerical setups considers $P=4$. The employed software is a streamlined implementation of the positive definite  QP dual algorithm from~\cite{goldfarb}  specifically coded in {\tt C} for this application to minimize any overheads.
The problem is treated as dense.  Typical CPU times reported in~\cref{cpuTime} are lower than the 10\,ms reported in~\cite{c1} and the 100\,ms reported in~\cite{c5} where similar online planning problems are addressed. They are also lower or equal to those reported in~\cite{c2,c7} where fixed-time online planning problems are solved. They are consistent with this objective and the hardware specifications of Atalante.
\begin{figure}[thpb]
   \begin{center}
      \begin{tabular}{ |c|c|c|c| } 
         \hline
         & min & max & mean \\ 
         \hline
         CPU time & 0.07\,\textrm{ms} & 0.75\,\textrm{ms} & 0.2\,\textrm{ms} \\
         \hline
      \end{tabular}
   \end{center}
   \caption{CPU time for the proposed algorithm (with $N=10$ maximum number of iterations , $4P+4=20$ variables, on a Ryzen 7 $1.7$\,\textrm{GHz} without turbo boost).}
   \label{cpuTime}
\end{figure}
\section{Conclusion}
In this paper, we presented a fast replanning algorithm for an exoskeleton with a patient. The method is applicable to general bipedal robots undergoing high-frequency velocity changes. Extensive numerical evaluation stresses its effectiveness and the safety increase. Thanks to our algorithm, the fall rate drops from $30\%$ (when using naive time scaling) to only $8\%$. Finally, we discussed the implementability of our algorithm on the Atalante onboard computer, providing evidence that the performance of our algorithm should be enough to run in real-time.

Future work will include implementation and experimental validation on Atalante. We expect some degradation of the stability because of model discrepancies, especially due to uncertainties in the patient model. Hence, we will also work on closing the gap between the simulation and experimental results.


\begin{thebibliography}{99}
   \bibitem{b01} M. Vukobratovic, D. Hristic and Z. Stojiljkovic, "Development of active anthropomorphic exoskeletons," Med Biol Eng, pp. 66-80, Jan. 1974.
  
   \bibitem{b02} A. Dollar and H. M. Herr, "Lower Extremity Exoskeletons and Active Orthoses: Challenges and State-of-the-Art," IEEE Tr. on Robotics, Special Issue on Biorobotics, vol. 24, no. 1, pp. 144-158, 2008.
  
   \bibitem{b03} D. P. Ferris, G. S. Sawicki and A. Domingo, "Powered lower limb orthoses for gait rehabilitation," Topics in spinal cord injury rehabilitation, pp. 34–49, 2005.
  
   \bibitem{b04} S. K. Banala, S. H. Kim, S. K. Agrawal, and J. P. Scholz, "Robot assisted gait training with active leg exoskeleton (ALEX)," IEEE Tr. on neural systems and rehabilitation engineering, vol. 17, pp. 2–8, 2009.
  
 
   \bibitem{b06} M. Bernhardt, G. Colombo and R. Riener, "Hybrid force-position control yields cooperative behaviour of the rehabilitation robot LOKOMAT," Proceedings of the 2005 IEEE 9th International Conference on Rehabilitation Robotics, 2005, pp. 536-539.
 
   \bibitem{b07} E. van Asseldonk and H. Kooij, "Robot-aided gait training with LOPES," Neurorehabilitation Technology, pp. 379, 2012.
 
   \bibitem{b08} S. T. Alan, "Control and trajectory generation of a wearable mobility exoskeleton for spinal cord injury patients," Doctoral dissertation, University of California, Berkeley, 2011.
  
   \bibitem{b09} S. K. Ann, "Development of a human machine interface for a wearable exoskeleton for users with spinal," Doctoral dissertation, University of California, Berkeley, 2011.
 
   \bibitem{b010} R. J. Farris, H. A. Quintero and M. Goldfarb, "Preliminary evaluation of a powered lower limb orthosis to aid walking in paraplegic individuals," IEEE Tr. on Neural Systems and Rehabilitation Engineering, vol. 19, pp. 652-659, 2011.
  
   \bibitem{b011} G. Zeilig, H. Weingarden, M. Zwecker, et al, "Safety and tolerance of the ReWalk exoskeleton suit for ambulation by people with complete spinalcord injury: A pilot study," The Journal of Spinal Cord Medicine, vol. 35(2), pp. 96-101, 2012.
  
 
 
   \bibitem{b014} O. Harib, A. Hereid, A. Agrawal, T. Gurriet, S. Finet, G. Boeris, A. Duburcq, M. Mungai, M. Masselin, A. Ames, K. Sreenath and J. Grizzle, "Feedback Control of an Exoskeleton for Paraplegics: Toward Robustly Stable, Hands-Free Dynamic Walking," IEEE Control Systems, vol. 38, 2018.
   
   \bibitem{kajita2003} S. Kajita et al., "Biped walking pattern generation by using preview control of zero-moment point," 2003 IEEE International Conference on Robotics and Automation (Cat. No.03CH37422), vol.2, pp. 1620-1626, 2003.
  
   \bibitem{wieber2018} P. B. Wieber, "Model Predictive Control for biped walking," Humanoid Robotics: A Reference, Springer Netherlands, pp.1077-1097, 2018.

\bibitem{hargraves-paris-87} C. R. Hargraves and S. W. Paris, "Direct trajectory optimization using nonlinear programming and collocation," Journal of Guidance, Control, and Dynamics 1987.

\bibitem{c1}
S. Caron and A. Kheddar, "Dynamic walking over rough terrains by nonlinear predictive control of the floating-base inverted pendulum," 2017 IEEE/RSJ International Conference on Intelligent Robots and Systems (IROS), 2017, pp. 5017-5024.

\bibitem{c2}
P. Fernbach, S. Tonneau, O. Stasse, J. Carpentier and M. Taïx, "C-CROC: Continuous and Convex Resolution of Centroidal Dynamic Trajectories for Legged Robots in Multicontact Scenarios," in IEEE Tr. on Robotics, vol. 36, no. 3, pp. 676-691, June 2020.

\bibitem{c4}
R. Tedrake, S. Kuindersma, R. Deits and K. Miura, "A closed-form solution for real-time ZMP gait generation and feedback stabilization," 2015 IEEE-RAS 15th International Conference on Humanoid Robots (Humanoids), 2015, pp. 936-940.

\bibitem{c5}
B. Ponton, M. Khadiv, A. Meduri and L. Righetti, "Efficient Multicontact Pattern Generation with Sequential Convex Approximations of the Centroidal Dynamics," in IEEE Tr. on Robotics, vol. 37, no. 5, pp. 1661-1679, Oct. 2021.


\bibitem{c7}
S. Caron and A. Kheddar, "Multi-contact Walking Pattern Generation based on Model Preview Control of 3D COM Accelerations," Humanoids, Nov 2016, Cancún, Mexico. pp.550-557

\bibitem{c8}
P. Hermanns and N. Thoai "Global optimization algorithm for solving bilevel programming problems with quadratic lower levels," Journal of Industrial and Management Optimization, 2010. 

\bibitem{cWieber}
P. B. Wieber, R. Tedrake and S. Kuindersma, "Modeling and Control of Legged Robots," In: B. Siciliano, O. Khatib (eds) Springer Handbook of Robotics. Springer Handbooks. Springer, Cham. 2016.

\bibitem{JolyVG}
L. D. Joly and C. Andriot, "Imposing motion constraints to a force reflecting telerobot through real-time simulation of a virtual mechanism," Proceedings of 1995 IEEE International Conference on Robotics and Automation,  pp. 357-362 vol.1, 1995.

\bibitem{liberzon}
D. Liberzon,  "Calculus of Variations and Optimal Control Theory: A Concise Introduction," Princeton University Press, 2012.


\bibitem{stairClimbing} S. Caron, A. Kheddar and O. Tempier, "Stair Climbing Stabilization of the HRP-4 Humanoid Robot using Whole-body Admittance Control," 2019 International Conference on Robotics and Automation (ICRA), pp. 277-283, 2019.

\bibitem{implicitFTH} S. G. Krantz and H. R. Parks, "The implicit function theorem: History, theory, and applications," Birkh\"auser Boston, Inc., Boston, MA, 2002.

\bibitem{goldfarb} D. Goldfarb and A. Idnani,  "A numerically stable dual method for solving strictly convex quadratic programs," Mathematical Programming 27, pp. 1–33, 1983.

\end{thebibliography}
\end{document}